\documentclass{amsart}
\usepackage{parskip}
\usepackage{enumerate}
\usepackage{amsmath}
\usepackage{amssymb}
\usepackage{amsthm}
\usepackage{graphicx}
\usepackage{mathabx}
\usepackage{float}
\usepackage{tikz}
\usepackage{etex}
\DeclareGraphicsExtensions{.bmp,.png,.pdf,.jpg}
\usepackage{latexsym}
\usepackage{amscd}
\usepackage[all]{xy}
\usepackage{amsmath}
\usepackage[utf8]{inputenc}
\DeclareMathOperator{\diam}{diam}

\DeclareMathOperator{\Cov}{Cov}

\newtheorem{teor}{Theorem}[section]
\newtheorem{cor}[teor]{Corollary}
\newtheorem{lem}[teor]{Lemma}

\newtheorem{prop}[teor]{Proposition}
\newtheorem{defn}[teor]{Definition}

\def\R{\mathbb{R}}

\def\E{\mathbb{E}}
\def\Z{\mathbb{Z}}
\def\C{\mathbb{C}}

\def\N{\mathbb{N}}

\def\LL{\mathcal{L}}

\newcommand{\PP}{\mathcal{P}}
\newcommand{\QQ}{\mathcal{Q}}

\newcommand{\NN}{\mathcal{N}}
\renewcommand{\S}{\mathcal{S}}

\newcommand{\GG}{\mathcal{G}}
\newcommand{\FF}{\mathcal{F}}
\newcommand{\UU}{\mathcal{U}}
\newcommand{\MM}{\mathcal{M}}
\newcommand{\JJ}{\mathcal{J}}
\newcommand{\EE}{\mathcal{E}}

\newcommand{\eps}{\varepsilon}

\newcommand{\disp}{\displaystyle}

\newcommand{\dmu}{\,d\mu}

\providecommand{\abs}[1]{\lvert#1\rvert}

\usepackage[utf8]{inputenc}
\usepackage[backref=page,colorlinks=true, linkcolor=blue, citecolor=blue, urlcolor=blue]{hyperref}
\usepackage[capitalise]{cleveref} 
\makeatletter

\def\MRbibitem{\@ifnextchar[\my@lbibitem\my@bibitem}

\def\mybiblabel#1#2{\@biblabel{{\hyperref{http://www.ams.org/mathscinet-getitem?mr=#1}{}{}{#2}}}}

\def\myhyperanchor#1{\Hy@raisedlink{\hyper@anchorstart{cite.#1}\hyper@anchorend}}

\def\my@lbibitem[#1]#2#3#4\par{%
  \item[\mybiblabel{#2}{#1}\myhyperanchor{#3}\hfill]#4%
  \@ifundefined{ifbackrefparscan}{}{\BR@backref{#3}}%
  \if@filesw{\let\protect\noexpand\immediate
    \write\@auxout{\string\bibcite{#3}{#1}}}\fi\ignorespaces%
}

\def\my@bibitem#1#2#3\par{%
  \refstepcounter\@listctr
  \item[\mybiblabel{#1}{\the\value\@listctr}\myhyperanchor{#2}\hfill]#3%
  \@ifundefined{ifbackrefparscan}{}{\BR@backref{#2}}%
  \if@filesw\immediate\write\@auxout
    {\string\bibcite{#2}{\the\value\@listctr}}\fi\ignorespaces%
}

\makeatother

\title[]{On the structure of the Birkhoff-irregular set for Subshifts of Finite Type}
\author{Sebastian Burgos }
\date{}

\begin{document}

\begin{abstract}
    We study the set of irregular points for topologically mixing subshifts of finite type. It is well known that despite the irregular set having zero measure for every invariant measure, it has full topological entropy and full Hausdorff dimension. We establish that the irregular set is not only abundant in terms of its dimensional properties, but also contains uncountably many pairwise disjoint invariant subsets, each of them dense and carrying full entropy and dimension. This result deepens our understanding of the complexity of irregular points in dynamical systems, highlighting their intricate structure and suggesting avenues for further exploration in related areas.
\end{abstract}

\maketitle
\section{Introduction}

Given a compact metric space $(X,d)$ and a continuous function $f\colon X\to X$, a point $x_0\in X$ is called \textit{(Birkhoff) irregular} if there is a continuous function $\varphi\colon X\to\R$ such that its Birkhoff average at $x_0$
$$\lim_{n\to\infty}\frac{1}{n}\sum_{j=0}^{n-1}\varphi(f^j(x_0))$$
does not exist. The \textit{irregular set of $\varphi$} is denoted by $\hat{X}(\varphi)$ and the \textit{irregular set of $X$} is denoted by $\hat{X}$. By Birkhoff Ergodic Theorem, for every $f$-invariant measure $\mu$ we have $\mu(\hat{X})=0$, so the irregular set is negligible from the measure-theoretic point of view. Thus, it is natural to measure its ``size" using different ``tools". We will focus on the most common ones with the first one being the Hausdorff dimension and the second one being the topological entropy; the former is geometric in nature while the latter is dynamic in nature. For topological entropy, we use its Carath\'eodory dimensional definition (see Section \ref{DIMandENT}), where it is not required that the set is compact or invariant.

Pesin and Pitskel \cite{PesinPitskel} studied the irregular set for the full shift on two symbols, and showed that it has full topological entropy. Barreira and Schmeling \cite{BS} showed that for subshifts of finite type, the irregular set carries full topological entropy and full Hausdorff dimension, and applied these results to uniformly hyperbolic dynamical systems. Thompson \cite{T1} studied the irregular set for any topological dynamical system satisfying the \textit{specification property}, and showed that it has full topological pressure with respect to any continuous function, in particular, full topological entropy. Later Thompson \cite{T2} showed that the irregular set for systems with the \textit{almost specification property} has full topological entropy and full Hausdorff dimension. Gelfert, Pacifico and Sanhueza \cite{GPS} gave sufficient conditions for a dynamical system to have irregular set with full topological entropy, applied to examples including nonuniformly hyperbolic surface diffeomorphisms with positive entropy and rational functions on the Riemann sphere. These results give us the property that even though the irregular set is negligible from the viewpoint of invariant measures, many times it is expected to have the largest possible size in terms of Hausdorff dimension and entropy.

In this paper we further advanced these results and reveal a highly sophisticated structure of the irregular set, namely, we show that the irregular set can be split in \emph{uncountably} many invariant subsets each of which carries full Hausdorff dimension and full topological entropy. Our main result is as follows.

\begin{teor}\label{mainthm}
    Let $(X,\sigma)$ be a topologically mixing subshift of finite type. Then, there is an uncountable collection $\NN$ of Borel measurable subsets of $X$ satisfying:
    \begin{enumerate}
        \item There is $\varphi\in C(X)$ such that every $N\in\NN$ is contained in $\hat{X}(\varphi)$;
        \item every $N\in\NN$ is dense in $X$ and satisfies $\sigma(N)\subset N$;
        \item for every distinct $N,N'\in\NN$ we have $N\cap N'=\emptyset$;
        \item for every $N\in\NN$ we have $\dim_HN=\dim_HX$ and $h(\sigma|N)=h_{top}(\sigma)$.
    \end{enumerate}
\end{teor}

Unlike previous work that are mentioned above, we follow the original approach by Pesin and Pitskel \cite{PesinPitskel}, which is to construct a bi-Lipschitz homeomorphism of $X$ which moves Birkhoff generic trajectories to Birkhoff irregular trajectories. However, we extend the setting from the basic full shift on two symbols considered in \cite{PesinPitskel} to more general subshifts of finite type. In order to achieve this we make substantial modifications to the approach in \cite{PesinPitskel}. Since horseshoes and iterated functions systems can be coded using full shifts and conjugacies preserve entropy, we obtain the following corollaries as direct consequences of Theorem \ref{mainthm}.

\begin{cor}\label{Cor1}
    Let $f\colon\Lambda\to\Lambda$ be a Smale horseshoe map. Then, there is an uncountable collection $\NN$ of Borel measurable subsets of $\Lambda$ satisfying:
    \begin{enumerate}
        \item There is $\varphi\in C(\Lambda)$ such that every $N\in\NN$ is contained in $\hat{\Lambda}(\varphi)$;
        \item every $N\in\NN$ is dense in $\Lambda$ and satisfies $f(N)\subset N$;
        \item for every distinct $N,N'\in\NN$ we have $N\cap N'=\emptyset$;
        \item for every $N\in\NN$ we have $h(f|N)=h_{top}(f)$.
    \end{enumerate}
\end{cor}

\begin{cor}\label{Cor2}
    Let $X\subset\R^d$ be the attractor of an iterated function system $(T_1,\ldots, T_m)$ satisfying the strong separation condition. Consider the dynamics $f\colon X\to X$ defined by the inverses of the maps $T_i$. Then, there is an uncountable collection $\NN$ of Borel measurable subsets of $X$ satisfying:
    \begin{enumerate}
        \item There is $\varphi\in C(X)$ such that every $N\in\NN$ is contained in $\hat{X}(\varphi)$;
        \item every $N\in\NN$ is dense in $X$ and satisfies $f(N)\subset N$;
        \item for every distinct $N,N'\in\NN$ we have $N\cap N'=\emptyset$;
        \item for every $N\in\NN$ we have $h(f|N)=h_{top}(f)$.
    \end{enumerate}
\end{cor}

\textit{Acknowledgements.} I would like to express my sincere appreciation to my advisor, Yakov Pesin, for his invaluable guidance and support throughout this research. I also wish to thank Scott Schmieding and Federico Rodriguez-Hertz for the comments and insights that contributed to this work. Additionally, I am deeply grateful to Jairo Bochi, whose input in our discussions was very helpful in generalizing the previous version of this paper. Finally, I would like to thank Sebastian Pavez, Ygor de Jesus, and the Penn State Student Dynamics Group for their encouragement and helpful comments.

\section{Hausdorff dimension and Topological entropy}\label{DIMandENT}

We follow the definitions of Hausdorff dimension and topological entropy from \cite{PesinDim}, using Carath\'eodory structures.

\subsection{Carath\'eodory Structures}
Let $X$ be a set and let $I$ be a set of indices. Let $\FF=\{U_i:i\in I\}$ be a collection of subsets of $X$. Assume that there exist functions $\xi,\psi,\eta\colon I\to[0,\infty)$ satisfying:
\begin{enumerate}
    \item[\textbf{C1.}] there is $i_0\in I$ such that $U_{i_0}=\emptyset$; if $U_i=\emptyset$ then $\psi(i)=\eta(i)=0$; if $U_i\neq\emptyset$ then $\psi(i)>0$ and $\eta(i)>0$; 
    \item[\textbf{C2.}] for every $\delta>0$ there is $\eps>0$ such that $\eta(i)\leq\delta$ for any $i\in I$ with $\psi(i)\leq\eps$;
    \item[\textbf{C3.}] for any $\eps>0$ there is a finite or countable collection $J\subset I$ that covers $X$ (i.e. $X\subset\bigcup_{i\in J}U_j$) and $\psi(J):=\sup\{\psi(i):i\in J\}\leq\eps$.
\end{enumerate}

We say that the tuple $\tau=(\FF,\xi,\eta,\psi)$ is a \textit{Carath\'eodory dimension structure} (or \textit{C-structure}).

For $\alpha\in\R$ and $Z\subset X$, the expression
\begin{equation}\label{measHdef}
    m_C(Z,\alpha):=\lim_{\eps\to 0}\inf\left\{\sum_{i\in J}\xi(i)\eta(i)^{\alpha}:J\subset I\text{ covers }Z,\psi(J)\leq\eps\right\}
\end{equation}
defines the \textit{$\alpha$-Carat\'eodory outer measure} on $X$, which induces a $\sigma$-additive measure on $X$. It can be shown \cite[Proposition 1.2]{PesinDim} that given $Z\subset X$, there is a critical value $-\infty\leq\alpha_{\star}\leq\infty$ such that $m_C(Z,\alpha)=\infty$ for every $\alpha<\alpha_{\star}$ and $m_C(Z,\alpha)=0$ for every $\alpha>\alpha_{\star}$.

\begin{defn}
    The Carath\'eodory dimension of a set $Z\subset X$ is defined by
    \begin{equation}
        \dim_CZ:=\inf\{\alpha:m_C(Z,\alpha)=0\}=\sup\{\alpha:m_C(Z,\alpha)=\infty\}.
    \end{equation}
\end{defn}

This Carath\'eodory dimension satisfies the following properties \cite[Theorem 1.1]{PesinDim}.

\begin{prop}\label{propDimC}
    The following hold:
    \begin{enumerate}
        \item $\dim_C\emptyset\leq 0$;
        \item if $Z_1\subset Z_2\subset X$, then $\dim_CZ_1\leq\dim_CZ_2$;
        \item $\dim_C\left(\displaystyle\bigcup_{i\in\N}Z_i\right)=\displaystyle\sup_{i\in\N}\dim_CZ_i$.
    \end{enumerate}
\end{prop}

Now suppose $(X,\mu)$ is a probability space such that $\FF$ is contained in the measurable sets of $X$. We define the \textit{Carath\'eodory dimension of $\mu$} by
\begin{equation}
    \dim_C\mu:=\inf\{\dim_CZ:\mu(Z)=1\}.
\end{equation}

\subsection{Hausdorff Dimension}
Let $(X,d)$ be a complete separable metric space and let $\FF=I$ be the set of all open subsets of $X$. Consider $\xi,\psi,\eta\colon\FF\to[0,\infty)$ defined by
\begin{equation}\label{StrucDim}
\xi(U)=1,\qquad\psi(U)=\eta(U)=\diam(U).
\end{equation}

\begin{defn}
Let $Z\subset X$. The Carath\'eodory dimension of $Z$ given by the $C$-structure $\tau=(\FF,\xi,\eta,\psi)$ defined by \eqref{StrucDim} is called the Hausdorff dimension of $Z$, and is denoted $\dim_HZ$.
\end{defn}

Assume now that $X$ is a Besicovitch metric space, then we have a way of computing the Hausdorff dimension of measures \cite[Theorem 7.1 and Appendix I]{PesinDim}. Given a Borel probability measure $\mu$ on $X$, its \textit{pointwise dimension} at $x\in X$ is defined by
\begin{equation}\label{pointdim}
    d_{\mu}(x):=\lim_{r\to 0}\dfrac{\log\mu(B(x,r))}{\log r},
\end{equation}
whenever the limit exists.

\begin{prop}\label{PropHdimMeas}
    If $d_{\mu}(x)=d$ for $\mu$-almost every $x\in X$, then $\dim_H\mu=d$.
\end{prop}

We also have the following proposition as immediate corollary of Proposition \ref{propDimC} (see also \cite[Proposition 3.3]{F}).

\begin{prop}\label{propDimH}
    The following hold:
    \begin{enumerate}
        \item $\dim_H\emptyset= 0$;
        \item if $Z_1\subset Z_2\subset X$, then $\dim_HZ_1\leq\dim_HZ_2$;
        \item $\dim_H\left(\displaystyle\bigcup_{i\in\N}Z_i\right)=\displaystyle\sup_{i\in\N}\dim_HZ_i$;
        \item if $L\colon X\to X$ is a bi-Lipschitz map, then $\dim_H L(Z)=\dim_HZ$ for every $Z\subset X$.
    \end{enumerate}
\end{prop}

\subsection{Topological entropy}\label{DefEntropy}
Let $(X,d)$ be a compact metric space and let $f\colon X\to X$ be a continuous map. Let $\UU$ be an open cover of $X$ and let
$$\S=\S_m(\UU):=\{{\bf U}=U_{i_0}\cdots U_{i_{m-1}}:U_{i_j}\in\UU\}$$
be the set of all strings of length $m$ of elements of $\UU$. Denote by $m({\bf U})$ the length of the string ${\bf U}$ and set $\S(\UU):=\bigcup_{m\geq 0}\S_m(\UU)$, which will be our set of indices. 

To a string ${\bf U}\in\S(\UU)$ we assign the set
\begin{equation}
    X({\bf U})=\{x\in X:f^j(x)\in U_{i_j}\text{ for }j=0,\ldots,m({\bf U})-1\}.
\end{equation}

Consider the collection of subsets $\FF=\FF(\UU)=\{X({\bf U}):{\bf U}\in\S(\UU)\}$ and define functions $\xi,\psi,\eta\colon\S(\UU)\to[0,\infty)$ by
\begin{equation}
    \xi({\bf U}):=1,\qquad\eta({\bf U}):=e^{-m({\bf U})},\qquad\psi({\bf U})=\frac{1}{m({\bf U})}.
\end{equation}

These functions determine a $C$-structure $\tau=\tau(\UU)=(\S,\FF,\xi,\eta,\psi)$. Observe that the Carath\'eodory outer measure, and hence the Carath\'eodory dimension, depend on the open cover $\UU$. Given $Z\subset X$, we denote its Carath\'eodory dimension given by this structure by $h_{\UU}(f|Z):=\dim_{C,\UU}Z$.

Let $\diam(\UU):=\sup\{\diam(U):U\in\UU\}$ be the diameter of the cover $\UU$. Then, Theorem 11.1 in \cite{PesinDim} allows us to make the following definition.

\begin{defn}
    Let $Z\subset X$. The topological entropy of $f$ restricted to the set $Z$ is defined by
    \begin{equation}
        h(f|Z):=\lim_{\diam(\UU)\to 0}h_{\UU}(f|Z).
    \end{equation}
\end{defn}

We have the corresponding proposition for topological entropy.

\begin{prop}\label{propDimE}
    The following hold:
    \begin{enumerate}
        \item $h(f|\emptyset)=0$;
        \item if $Z_1\subset Z_2\subset X$, then $h(f|Z_1)\leq h(f|Z_2)$;
        \item $h\left(f|\bigcup_{i\in\N}Z_i\right)=\displaystyle\sup_{i\in\N}h(f|Z_i)$.
    \end{enumerate}
\end{prop}

It is important to mention that for a compact invariant set $Z$, the topological entropy defined using this Carath\'eodory structure coincides with the classical topological entropy.

\section{Subshifts of Finite Type}

\subsection{Basic definitions}
Let $A$ be a $d\times d$ matrix with entries in $\{0,1\}$. The compact set
$$X=\{(x_i)_{i\geq 0}:x_i\in\{0,1,\ldots d-1\}\text{ and } A(x_i,x_{i+1})=1\text{ for every } i\}$$
is called a \textit{topological Markov shift}, and $A$ is called its \textit{adjacency (or transfer) matrix}. The \textit{shift map} $\sigma\colon X\to X$ is defined by $\sigma((x_i)_{i\geq0})=(x_{i+1})_{i\geq 0}$. The dynamical system $(X,\sigma)$ is called a (one-sided) \textit{subshift of finite type}. For $x=(x_i)_{i\geq 0}$ and $y=(y_i)_{i\geq 0}$, define a distance on $X$ by 
\begin{equation}\label{SymbolicMetric}
    d(x,y):=e^{-\min\{i\geq 0:x_i\neq y_i\}}
\end{equation}
when $x\neq y$, and $d(x,y)=0$ when $x=y$.

The matrix $A$ is called \textit{aperiodic} if there is $k\in\N$ such that all the entries of $A^k$ are positive. This is equivalent to the dynamical system $(X,\sigma)$ being topologically mixing. This is, for any two open sets $U,V\subset X$, there is $N\in\N$ such that $\sigma^{n}(U)\cap V\neq\emptyset$ for every $n\geq N$.

Let $\LL_n(X)$ be the set of words in the alphabet $\{0,1,\ldots,d-1\}$ which have length $n$. Given a word $w\in\LL_n(X)$, the \textit{cylinder set centered at $w$} is defined by
$$[w]:=\{x\in X:x_0=w_0,\ldots,x_{n-1}=w_{n-1}\}.$$

Given $x\in X$, denote $x|_{i}^j=x_ix_{i+1}\cdots x_j$ and define the \textit{cylinder of length $n$ centered at $x$} by $C_n(x):=[x|_0^{n-1}]$, which coincide with the open ball centered at $x\in X$ of length $e^{-n}$. These sets generate the topology on $X$ induced by the metric $d$ defined in \eqref{SymbolicMetric}.

The following theorem is well known, and will be very important for us describing the measure of maximal entropy of $(X,\sigma)$.
\begin{teor}[Perron-Frobenius Theorem]\label{PFT}
    Let $A$ be a non-negative aperiodic $d\times d$ matrix. Then, the following hold:
    \begin{enumerate}
        \item There is a positive real eigenvalue $\lambda$ of $A$ such that all other eigenvalues $\rho\in\C$ of $A$ satisfy $\abs{\rho}<\lambda$.

        \item The eigenvalue $\lambda$ is simple. That is, the corresponding eigenspace is one-dimensional.

        \item There are unique left- and right- eigenvectors ${\bf u}$ and ${\bf v}$ of $A$ corresponding to the eigenvalue $\lambda$ with ${\bf u}(i)>0$, ${\bf v}(i)>0$ for all $i$, also satisfying $\sum_i{\bf u}(i)=1$ and ${\bf u}^T{\bf v}=1$.
    \end{enumerate}
\end{teor}

\subsection{Markov measures}

A $d\times d$ matrix $P$ is called \textit{stochastic} if its entries are between 0 and 1, and the elements of each of its rows add up to 1. A vector ${\bf p}\in\R^n$ is called a \textit{probability vector} if it has coordinates between 0 and 1, and their sum equals 1. A stochastic matrix $P$ is \textit{compatible} with $X$ if $P(i,j)>0$ only as long as $A(i,j)>0$. If $P$ is aperiodic, there is a unique probability vector ${\bf p}$ satisfying ${\bf p}^TP={\bf p}$.

\begin{defn}
    The Markov measure $\mu_P$ associated to $(P,{\bf p})$ is defined on cylinders as follows. For $w\in\LL_n(X)$, let
    \begin{equation}\label{Markovmeas}
        \mu_P([w]):={\bf p}(w_0)P(w_0,w_1)P(w_1,w_2)\cdots P(w_{n-2},w_{n-1}).
    \end{equation}
\end{defn}

Markov measures are ergodic $\sigma$-invariant Borel probability measures on $X$. Moreover, these measures have exponential decay of correlations with respect to H\"older observables (see for example \cite[Theorem 7.4.1]{VO}), i.e. there is $\gamma>0$ such that for every $\phi,\psi\colon X\to\R$ H\"older continuous, we have
$$\Big|\int\phi(\psi\circ\sigma^n)\dmu-\int\phi\dmu\int\psi\dmu\Big|\leq Ce^{-\gamma n}$$
for all $n\in\N$, where $C=C(\phi,\psi)$.

The measure-theoretic entropy of a Markov measure $\mu_P$ is given by
\begin{equation}\label{entropyMarkov}
h(\mu_P)=\sum_{i=0}^{d-1}{\bf p}(i)\sum_{j=0}^{d-1}-P(i,j)\log P(i,j).
\end{equation}

\subsection{Entropy, dimension and the Parry measure}

We will now state some facts about Hausdorff dimension of ergodic measures, as well as entropy and Hausdorf dimension of the entire space $X$. Recall that we are considering the metric defined by \eqref{SymbolicMetric}. We denote by $\MM_{\sigma}(X)$ the set of all $\sigma$-invariant probability measures on $X$, and for any $\mu\in\MM_{\sigma}(X)$ we denote by $h(\mu)$ its measure theoretical entropy.

The first proposition is an immediate corollary of Shannon-McMillan Breiman theorem and Proposition \ref{PropHdimMeas}. The second proposition can be found in \cite[Theorem A.2.9]{PesinDim}.

\begin{prop}\label{dim=entropy}
    Let $(X,\sigma)$ be a subshift of finite type. Then, for any ergodic measure $\mu$ on $X$ we have $\dim_H\mu=h(\mu)$.
\end{prop}

\begin{prop}
    Let $(X,\sigma)$ be a subshift of finite type with adjacency matrix $A$. Then, $\dim_HX=h_{top}(\sigma)=\log\lambda$, where $\lambda$ is as in Theorem \ref{PFT}.
\end{prop}

The well known Variational Principle states that
$$h_{top}(\sigma)=\sup\{h(\mu):\mu\in\MM_{\sigma}(X)\},$$

where any measure that attains the supremum (if any) is called a \textit{measure of maximal entropy of $(X,\sigma)$}. Parry \cite{parry} showed that if $(X,\sigma)$ is a topologically mixing subshift of finite type, then there is an aperiodic stochastic matrix $P$ such that $m:=\mu_P$ is the unique measure of maximal entropy for $(X,\sigma)$. Such measure is called \textit{the Parry measure}. The matrix $P$ and its left eigenvector ${\bf p}$ of the Parry measure are explicitly given by
\begin{equation}\label{Parryentries}
    P(i,j)=A(i,j)\frac{{\bf v}(j)}{\lambda {\bf v}(i)},\qquad {\bf p}(i)={\bf u}(i){\bf v}(i),
\end{equation}
where $\lambda$, ${\bf u}$, ${\bf v}$ are as in Theorem \ref{PFT}.

The Parry measure also has the following characteristic, which can be shown using $\eqref{Parryentries}$: for any $i,j\in\{0,\ldots d-1\}$ and $w\in\LL_{n-1}(X)$ such that $iwj\in\LL_{n+1}(X)$ it holds
\begin{equation}\label{parryofcylinder}
    m([iwj])=\frac{{\bf u}(i){\bf v}(j)}{\lambda^n}.
\end{equation}

\subsection{Equilibrium states of some locally constant potentials}

In this section we will discuss a generalization of the Parry measure, which are equilibrium states for potentials that depend on the first two symbols. Let $\phi\colon X\to\R$ be a function such that $\phi(x)=\phi(y)$ whenever $x_0x_1=y_0y_1$. We will denote by $\phi(a,b)$ the value of $\phi|_{[ab]}$ for $ab\in\LL_2(X)$. It is known that there is a unique equilibrium state $\mu_{\phi}$ for $\phi$, which is a Markov measure determined by the aperiodic matrix $D_{\phi}$ defined by
\begin{equation}\label{DefDphi}
    D_{\phi}(a,b) ={\left\{\begin{array}{lcc}
        e^{\phi(a,b)} & , & ab\in\LL_2(X)\\
        0 & , & \text{otherwise}
    \end{array}\right.}.
\end{equation}

By Perron Frobenius Theorem, there are unique left and right maximal eigenvectors ${\bf u}_{\phi}$ and ${\bf v}_{\phi}$ associated to the maximal eigenvalue $\rho_{\phi}$, which can be chosen such that ${\bf u}_{\phi}^T{\bf v}_{\phi}=1$. The stochastic matrix $\Phi$ associated to Markov measure $\mu_{\phi}$ is related to $D_{\phi}$ by
\begin{equation}
    \Phi(a,b)=D_{\phi}(a,b)\frac{{\bf v}_{\phi}(b)}{\rho_{\phi}{\bf v}_{\phi}(a)}.
\end{equation}

Observe that this is the same relation we had for the Parry measure, where $\phi\equiv 0$ and hence $D_{\phi}=A$.

We obtain the following formula for the measure of a cylinder in terms of the matrix $D_{\phi}$ and the vectors ${\bf u}_{\phi},{\bf v}_{\phi}$. Given a word ${\bf i}=i_0\cdots i_{n-1}\in\LL_n(X)$, we have
\begin{equation}\label{measurecylinder}
    \mu_{\phi}([{\bf i}])=\frac{1}{\rho_{\phi}^{n-1}}{\bf u}_{\phi}(i_0)\left(\prod_{j=0}^{n-2}D_{\phi}(i_j,i_{j+1})\right){\bf v}_{\phi}(i_{n-1}).
\end{equation}

One can check that $\mu_{\phi}$ is the equilibrium state of $\phi$ by verifying that it satisfies the Gibbs property with $P(\phi)=\log\rho_{\phi}$.

\section{Proof of the main theorem}
The idea of the proof is to find a sequence of ergodic measures such that their entropies (resp. Hausdorff dimensions) approximate the maximal entropy (resp. Hausdorff dimension). Thus, the union of some of their regular sets will approximate the desired quantities. Following the ideas of \cite{PesinPitskel}, we will map these regular sets onto irregular sets without distorting the maximal entropy and dimension. By Proposition \ref{dim=entropy}, we can refer to only entropy of the measures, as the Hausdorff dimension coincide.

\subsection{Constructing a map that sends regular points to irregular points.}

We will consider two words $\xi,\eta$ of the form $\xi=wsw$ and $\eta=wtw$, where $\abs{w}>\abs{s}=\abs{t}$, satisfying that the only possible overlapping between these words can occur in the word $w$. Let $\mu$ be an ergodic measure such that $\mu([\xi])\neq\mu([\eta])$. Consider the continuous functions $\varphi,\psi\colon X\to\R$ given by $\varphi(x):=\chi_{[\xi]}(x)$ and $\psi(x):=\chi_{[\eta]}(x)$ and denote by $S_n\varphi,S_n\psi$ their respective $n$th Birkhoff sums. Let 
\begin{equation}\label{Bmu}
    B_{\mu}:=\left\{x\in X:\lim_{n\to\infty}\frac{1}{n}S_n\varphi(x)=\mu([\xi])\quad\text{and}\quad\lim_{n\to\infty}\frac{1}{n}S_n\psi(x)=\mu([\eta])\right\},
\end{equation}
and observe that $\mu(B_{\mu})=1$ by Birkhoff Ergodic Theorem. We now proceed to construct our words $\xi,\eta$.

Consider a periodic point $(\beta_1\beta_2\cdots\beta_{\kappa})^{\infty}$ of minimal period $\kappa$, so that $\beta_1\to\ldots\to\beta_{\kappa}\to\beta_1$ is a minimal cycle in the graph representing $X$. Since $(X,\sigma)$ is topologically mixing, there is a symbol $\alpha$ that does not belong to the cycle that can be followed by one of the $\beta_i$'s, say $\beta_1$ (otherwise relabel the $\beta_i$'s).

For $p\in\N$, let $w_p:=\alpha(\beta_1\cdots\beta_{\kappa})^p\in\LL_{\kappa p+1}(X)$. By topological mixing, we can choose two distinct words $s,t$ of the same length (say length $m$) that can connect $\beta_{\kappa}$ and $\alpha$. That is, $\beta_{\kappa}s\alpha,\beta_{\kappa}t\alpha\in\LL_{m+2}(X)$. Observe that we can increase $p$ such that $\abs{w_p}>m$. Also, by construction, no prefix of the word $w$ is also a suffix of $w$, so we conclude that the only possible overlapping between $\xi$ and $\eta$ are on their first and last $w$ pieces. Thus, the words $\xi,\eta$ satisfy what we need. The constructions of the words $s,t$ and the measure $\mu$ will be made precise later. For this part of the proof, these are the only properties we need. 

Let $M$ be the length of the words $\xi$ and $\eta$. Consider the map $L\colon X\to X$ defined as follows:

\begin{itemize}
    \item If $(2k)!\leq i\leq (2k+1)!$ for some $k\in\N$, set $L(x)_i:=x_i$.
    \item If $(2k-1)!+1\leq i\leq (2k)!-1-M$ for some $k\in\N$, set
    
    \begin{equation}
        L(x)\Big|_i^{i+M-1}:= \left\{\begin{array}{lcl}
                \eta & \text{if} & x|_i^{i+M-1}=\xi\\
                \\
                \xi & \text{if} & x|_i^{i+M-1}=\eta\\
                \\
                x|_i^{i+M-1} & \text{if} & x|_i^{i+M-1}\neq\xi,\eta\\
        \end{array}\right..
    \end{equation}
\end{itemize}

The properties of the words $\xi,\eta$ ensure that there are no ambiguities when defining the map $L$.

\begin{prop}\label{Lirreg}
    The map $L$ satisfies $L(B_{\mu})\subset\hat{X}(\varphi)$.
\end{prop}

\begin{proof}
    Fix $x\in B_{\mu}$. Observe that if $(2n-1)!+1\leq i \leq (2n)!-1-M$ for big enough $n\in\N$, then
    \begin{equation}
        \varphi(\sigma^i(L(x)))=\left\{\begin{array}{ccl}
        1 & \text{if} & L(x)\Big|_{i}^{i+M-1}=\xi \\
        \\
        0 & \text{if} & L(x)\Big|_{i}^{i+M-1}\neq\xi 
    \end{array}\right.=\left\{\begin{array}{ccl}
        1 & \text{if} & x\Big|_{i}^{i+M-1}=\eta \\
        \\
        0 & \text{if} & x\Big|_{i}^{i+M-1}\neq\eta 
    \end{array}\right.=\psi(\sigma^i(x)).
    \end{equation}

    Hence,
    \begin{align*}
        \Big|\frac{1}{(2n+1)!}\sum_{i=0}^{(2n+1)!}\varphi(\sigma^i(L(x)))-\mu([\xi])\Big|&\leq\Big|\frac{1}{(2n+1)!}\sum_{i=0}^{(2n+1)!}\varphi(\sigma^i(L(x)))-\frac{1}{(2n+1)!}\sum_{i=0}^{(2n+1)!}\varphi(\sigma^i(x))\Big|\\ &+\Big|\frac{1}{(2n+1)!}\sum_{i=0}^{(2n+1)!}\varphi(\sigma^i(x))-\mu([\xi])\Big|\\
        &\leq\frac{(2n)!+M}{(2n+1)!}+\Big|\frac{1}{(2n+1)!}\sum_{i=0}^{(2n+1)!}\varphi(\sigma^i(x))-\mu([\xi])\Big|\longrightarrow 0
    \end{align*}
    as $n\to\infty$. Similarly,
    \begin{align*}
        \Big|\frac{1}{(2n)!}\sum_{i=1}^{(2n)!-1}\varphi(\sigma^i(L(x)))-\mu([\eta])\Big|&\leq\Big|\frac{1}{(2n)!}\sum_{i=1}^{(2n)!-1}\varphi(\sigma^i(L(x)))-\frac{1}{(2n)!}\sum_{i=1}^{(2n)!-1}\psi(\sigma^i(x))\Big|\\ &+\Big|\frac{1}{(2n)!}\sum_{i=1}^{(2n)!-1}\psi(\sigma^i(x))-\mu([\eta])\Big|\\
        &\leq\frac{(2n-1)!+M}{(2n)!}+\Big|\frac{1}{(2n)!}\sum_{i=1}^{(2n)!-1}\psi(\sigma^i(x))-\mu([\eta])\Big|\longrightarrow 0
    \end{align*}
    as $n\to\infty$. Since $\mu([\xi])\neq\mu([\eta])$, we have $L(x)\in\hat{X}(\varphi)$. The desired result follows as $x\in B_{\mu}$ was arbitrary.\qedhere
\end{proof}

Observe that $L=L^{-1}$, so in order to show that $L$ preserves Hausdorff dimension (using Proposition \ref{propDimH}) it suffices to show it is Lipschitz (hence automatically bi-Lipschitz).

\begin{prop}\label{LLipschitz}
    The map $L\colon X\to X$ is Lipschitz.
\end{prop}

\begin{proof}
    Let $x,y\in X$ be such that $d(x,y)=e^{-i}$. That is, $x_i\neq y_i$ and $x_j=y_j$ for every $0\leq j\leq i-1$. 

    \textit{Case 1: Suppose $(2k)!\leq i\leq(2k+1)!$ for some $k$.}

    In this interval the map $L$ makes no substitutions, and on previous indices the same substitutions were made for $x$ and $y$ so we obtain $d(L(x),L(y))=e^{-i}=d(x,y)$.\\

    \textit{Case 2: Suppose $(2k-1)!+1\leq i\leq(2k)!-1$ for some $k$.}

    If the interval is not long enough to admit the words $\xi,\eta$ inside, then $L$ makes no substitutions and we get the same conclusion as in Case 1. If this interval is long enough to admit words $\xi,\eta$ inside, but neither $x_i$ nor $y_i$ belongs to one of these words, then no substitutions are made and we obtain again the same conclusion as in Case 1.

    Assume that $x_i$ is inside the word $\xi$ (for the word $\eta$ the argument is analogous). Denote by $\Delta$ the interval of indices such that $i\in\Delta$ and $x|_{\Delta}=\xi$. We claim that $L(x)|_{\Delta}\neq L(y)|_{\Delta}$. If the interval $\Delta$ starts at an index less than $(2k-1)!+1$ or greater than $(2k)!-1-M$, then no substitutions are made to this interval on both $x$ and $y$, and hence $L(x)|_{\Delta}=x|_{\Delta}\neq y|_{\Delta}=L(y)|_{\Delta}$ (as $x_i\neq y_i$). If the interval $\Delta$ is fully contained in $\{(2k-1)!+1,\ldots(2k)!-1\}$ and $L(y)|_{\Delta}=L(x)|_{\Delta}=\eta$, then $y|_{\Delta}=\xi=x|_{\Delta}$, which contradicts $x_i\neq y_i$. Since for indices less than the beginning of $\Delta$ the same substitutions are made for $x$ and $y$, the sequences $L(x),L(y)$ must agree up to at least index $i-M$. Thus, we obtain
    $$d(L(x),L(y))\leq e^{-(i-M)}=e^Md(x,y)$$
    as desired.\qedhere
\end{proof}

\subsection{Contructing a sequence of measures approximating the maximal dimension and maximal entropy}

In this section we will construct a sequence of Markov measures $\mu_k$ whose Hausdorff dimensions (equivalently, entropies) converge to the Hausdorff dimension (equivalently, entropy) of the Parry measure. Since we need that our measures satisfy $\mu([\xi])\neq\mu([\eta])$, we need to choose carefully our words $s,t$.

First, consider a path of minimal length that connects the symbols $\beta_{\kappa}$ and $\alpha$, say $\beta_{\kappa}\gamma_1\gamma_2\cdots\gamma_{\ell}\alpha$, and observe that $\gamma_i\neq\alpha$ for all $i$ by minimality. By topological mixing, there is $n_0$ such that for any $n\geq n_0$ there is a word connecting the symbol $\alpha$ with itself. Let $e\in\N$ be the minimum number with $n:=\kappa e> n_0$ and consider an admissible path $\alpha\theta_1\theta_2\cdots\theta_{n-1}\alpha$. Now define the words $s:=(\beta_1\cdots\beta_{\kappa})^e\gamma_1\gamma_2\cdots\gamma_{\ell}$ and $t:=\gamma_1\gamma_2\cdots\gamma_{\ell}\alpha\theta_1\theta_2\cdots\theta_{n-1}$. Observe that $\abs{s}=\kappa e+\ell=\ell+n=\abs{t}$. Recall that now we can increase the length of our word $w$ ($w=w_p$ by choosing $p$ large) so that $\abs{w}>\abs{s}=\abs{t}$. We have
\begin{align*}
\xi&=\overbrace{\alpha(\beta_1\cdots\beta_{\kappa})^p}^{w}\overbrace{(\beta_1\cdots\beta_{\kappa})^e\gamma_1\gamma_2\cdots\gamma_{\ell}}^{s}\overbrace{\alpha(\beta_1\cdots\beta_{\kappa})^p}^{w},\\
\eta&=\overbrace{\alpha(\beta_1\cdots\beta_{\kappa})^p}^{w}\overbrace{\gamma_1\gamma_2\cdots\gamma_{\ell}\alpha\theta_1\theta_2\cdots\theta_{n-1}}^{t}\overbrace{\alpha(\beta_1\cdots\beta_{\kappa})^p}^{w}.
\end{align*}

Now let us consider the following locally constant potential $\phi\colon X\to\R$. Define $\phi(\alpha,\theta_1):=1$ and $\phi(a,b):=0$ for every $ab\in\LL_2(X)\setminus\{\alpha\theta_1\}$. Let $\mu_{\phi}$ be the equilibrium state of this potential. Since $\xi,\eta$ share the first and last symbols, and since $s$ does not contain the symbol $\alpha$ and $t$ contains the word $\alpha\theta_1$, by the formula \eqref{measurecylinder} it is not hard to see that $\mu_{\phi}([\xi])<\mu_{\phi}([\eta])$.

For $q>0$, let $\mu_q$ be the equilibrium state of the locally constant potential 
$$q\phi(x):=\left\{\begin{array}{lcc}
   q  & , & x_0x_1=\alpha\theta_1 \\
   0  & , & x_0x_1\neq\alpha\theta_1
\end{array}\right..$$

A similar argument as before shows that $\mu_q([\xi])<\mu_q([\eta])$ for every $q>0$. Since the $\mu_q$'s are Markov measures, and the stochastic matrices $\Phi_q$ and their corresponding probability vectors vary continuously in $q$, we have that $\disp\lim_{q\to 0^+}h(\mu_q)=h(m)=\log\lambda$. 

\subsection{Constructing a set $N$ with properties (1) and (2) in Theorem \ref{mainthm}} 

Consider a decreasing sequence $\PP=\{q_k:k\geq 1\}$ of positive numbers in $(0,1)$ converging to $0$. Let $\mu_k$ be the equilibrium state of $q_k\phi$.


For every $k\geq 1$ let $B_k:=B_{\mu_k}$ be as in equation \eqref{Bmu} (using the measure $\mu_k$ instead of $\mu$), and define $N:=\bigcup_{k\geq 1}L(B_k)$. By Proposition \ref{Lirreg}, for each $k\in\N$ we have $L(B_k)\subset\hat{X}(\varphi)$, so $N\subset\hat{X}(\varphi)$.

\begin{lem}\label{dense}
    The set $N$ is dense in $X$.
\end{lem}

\begin{proof}
    Since the vectors ${\bf u},{\bf v}$ in Theorem \ref{PFT} have positive entries, by \eqref{parryofcylinder} the Parry measure $m$ gives positive measure to any cylinder in $X$.

    Let $C$ be an arbitrary cylinder in $X$. By the previous observation, $m(C)>0$. By continuity of the expression \eqref{Markovmeas} depending on the matrices $\Phi_{q_k}$ and their probability vectors, there is $k_0\in\N$ such that $\mu_{k_0}(C)>0$. Thus, $B_{k_0}\cap C\neq\emptyset$ as $\mu_{k_0}(B_{k_0})=1$. We conclude that the set $\bigcup_{k\geq 1}B_k$ is dense in $X$. Since $L$ is a homeomorphism of $X$, the set $N=L(\bigcup_{k\geq 1}B_k)$ is also dense in $X$.
\end{proof}

\begin{lem}
    We have $\sigma(N)\subset N$.
\end{lem}
\begin{proof}
    Let $y\in N$. Then, there is $k\in\N$ and $x\in B_k$ such that $L(x)=y$. We need to show that $\sigma(y)\in N$, so it suffices to show that $L(\sigma(y))\in B_k$. 

    Let $n\in\N$ and let $m(n)$ be the maximum natural number $m$ such that $m!\leq n-1$. Observe that $m(n)\to\infty$ as $n\to\infty$. Among the first $n$ symbols, the number of times the word $\xi$ occurs in the sequences $x$ and $(L\circ\sigma\circ L)(x)$ may differ in at most $m(n)$. Then we have
    \begin{align*}
        \Big|\frac{1}{n}\sum_{i=0}^{n-1}\varphi(\sigma^i(x))-\frac{1}{n}\sum_{i=0}^{n-1}\varphi(\sigma^i((L\circ\sigma\circ L)(x)))\Big|\leq\frac{m(n)}{n}\leq\frac{1}{(m(n)-1)!}\to 0
    \end{align*}
    as $n\to\infty$. An analogous argument shows the same for $\psi$ (using the word $\eta$), so we conclude that $(L\circ\sigma\circ L)(x)=L(\sigma(y))$ belongs to $B_k$ and thus $\sigma(y)\in L(B_k)\subset N$ as desired.\qedhere
\end{proof}

\subsection{The set $N$ has full Hausdorff dimension}

\begin{lem}
    The set $N=\bigcup_{k\in\N}L(B_k)$ has Hausdorff dimension $\dim_HX$.
\end{lem}

\begin{proof}
    Since $L$ is bi-Lipschitz (Proposition \ref{LLipschitz}), it satisfies $\dim_HL(Z)=\dim_HZ$ for every $Z\subset X$ (see Proposition \ref{propDimH}). Since $\mu_k(B_k)=1$, by definition we have $\dim_H\mu_k\leq\dim_HB_k$. Recall that as $q_k\searrow 0$, the entropies $h(\mu_k)$ converge to $\sup_{k\geq 1}h(\mu_k)=\log\lambda$. Thus, using Proposition \ref{propDimH} and Proposition \ref{dim=entropy} we have
    \begin{align*}
    \dim_HX&\geq\dim_HN=\sup_{k\geq1}\dim_HL(B_k)=\sup_{k\geq1}\dim_HB_k\\
    &\geq\sup_{k\geq 1}\dim_H\mu_k=\sup_{k\geq 1}h(\mu_k)=\log\lambda=\dim_HX.\qedhere
\end{align*}
\end{proof}

\subsection{The set $N$ has full topological entropy}

The argument for topological entropy is almost the same as for Hausdorff dimension. The extra difficulty in this case is that the map $L$ does not necessarily preserves entropy, as it does not commute with the dynamics $\sigma$. However, we can fix this with the following result.

\begin{prop}\label{entropyineq}
For every $k\in\N$, we have $h(\sigma|L(B_k))\geq h(\mu_k)$.
\end{prop}

In order to prove Proposition \ref{entropyineq} we need the following lemma.

\begin{lem}\label{SLLN} Fix $m\in\N$. Then for $\mu_k$ almost every $x\in B_k$, we have
$$\lim_{n\to\infty}-\frac{1}{n}\log\mu_k(C_{n+m}(L(x)))=h(\mu_k).$$
\end{lem}

\begin{proof}
    In order to reduce some notation, let $\mu=\mu_k$ and denote by $Q=\Phi_{q_k\phi}$ its stochastic matrix and by ${\bf q}$ its probability vector. Observe that we have the following Birkhoff average-like expression
    \begin{align*}
    -\frac{1}{n}\log\mu(C_n(x))&=-\frac{1}{n}\log({\bf q}(x_0)Q(x_0,x_1)Q(x_1,x_2)\cdots Q(x_{n-2},x_{n-1}))\\
    &=-\frac{1}{n}\log {\bf q}(x_0)+\frac{1}{n}\sum_{i=0}^{n-2}-\log Q(x_i,x_{i+1})\\
    &=-\frac{1}{n}\log {\bf q}(x_0)+\frac{1}{n}\sum_{i=0}^{n-2}\rho(\sigma^i(x)),
    \end{align*}
    where $\rho(x)=-\log Q(x_0,x_1)$. We will use the idea from \cite{PesinPitskel} that splits the original expression into two sums. Let $S_1$ be the subset of natural numbers $i$ such that $(2k)!\leq i\leq (2k+1)!$ for some $k\in\N$ and let $S_2:=\N\setminus S_1$. Consider $S_1(n):=\{0,1,\ldots,n-1\}\cap S_1$ and $S_2(n):=\{0,1,\ldots,n-1\}\cap S_2$. Since $m$ is fixed and negligible for $n$ large, we write
    \begin{equation*}
        \lim_{n\to\infty}-\frac{1}{n}\log\mu(C_{n+m}(L(x)))=\lim_{n\to\infty}\frac{1}{n}\left(\sum_{i\in S_1(n)}\rho(\sigma^i(L(x)))+\sum_{i\in S_2(n)}\rho(\sigma^i(L(x)))\right). 
    \end{equation*}

\textit{Claim:} We have $\disp\lim_{n\to\infty}\frac{1}{\abs{S_1(n)}}\sum_{i\in S_1(n)}\rho(\sigma^i(L(x)))=h(\mu)$.

\begin{proof}[Proof of Claim:] Since we are considering indices in the set $S_1$ where the map $L$ does not change the symbols, the limit on the left hand side equals $\disp\lim_{n\to\infty}\frac{1}{\abs{S_1(n)}}\sum_{i\in S_1(n)}\rho(\sigma^i(x))$. In order to obtain the desired limit, we need to use the following result (see \cite{J,K}), which establishes the Strong Law of Large Numbers for \textit{quasi uncorrelated} variables. It will be written in probabilistic notation.

\begin{prop}[Janisch, Korchevsky] Let $(X_n)_n$ be a sequence of non-negative random variables on a probability space and let $S_n=X_1+\cdots+X_n$. Assume they satisfy:
\begin{enumerate}
    \item[(i)] each $X_n$ has finite variance $V(X_n)$, and
    $$\sup_{n\in\N}\frac{\E(S_n)}{n}<\infty\quad\text{and}\quad\sum_{n=1}^{\infty}\frac{V(X_n)}{n^2}<\infty;$$
    \item[(ii)] there exists a constant $c>0$ such that 
    $$V(S_n)\leq c\sum_{k=1}^nV(X_k)$$
    for every $n\in\N$.
\end{enumerate}
Then, the SLLN holds, i.e.
$$\lim_{n\to\infty}\frac{S_n-\E(S_n)}{n}=0$$
almost surely.
\end{prop}

We will show that the sequence $X_n(x):=-\log Q(x_{n-2},x_{n-1})$ for $n\in S_1\cup\{0\}$ satisfies assumptions $(i)$ and $(ii)$. Denote by $S^*_n:=\sum_{k\in S_1(n)}X_k$.

First observe that since the measure $\mu$ is $\sigma$-invariant and $X_n=\rho\circ\sigma^n$, we have $\E(X_n)=\E(X_0)$ and $V(X_n)=V(X_0)$. It is clear now that $(i)$ holds. For $(ii)$, we have

\begin{align*}
V(S^*_n)&=\sum_{k\in S_1(n)}V(X_k)+2\sum_{\{i,j\in S_1(n):i<j\}}\Cov(X_i,X_j)\\
&=\abs{S_1(n)}V(X_0)+2\sum_{\{i,j\in S_1(n):i<j\}}\Cov(X_i,X_j).
\end{align*}

Observe that for $i<j$, since $\sigma$ preserves $\mu$ we have
\begin{align*}
    \Cov(X_i,X_j)&=\E(X_iX_j)-\E(X_i)\E(X_j)\\
    &=\int(\rho\circ\sigma^i)(\rho\circ\sigma^j)\dmu-\int\rho\circ\sigma^i\dmu\int\rho\circ\sigma^j\dmu\\
    &=\int\rho(\rho\circ\sigma^{j-i})\dmu-\int\rho\dmu\int\rho\dmu\\
    &\leq Ce^{-\gamma(j-i)}
\end{align*}
 for some $\gamma>0$, consequence of the exponential decay of correlations of the Markov measure $\mu$. Here $C$ depends on $\rho$, which depends only on the matrix $Q$. 

 Now write $S_1(n):=A_1\cup A_2\cup\cdots\cup A_{k_n}$, where $k_n=\max\{k\in\N:(2k)!\leq n\}$, $A_{\ell}=\{(2\ell)!,\ldots,(2\ell+1)!\}$ for $1\leq \ell\leq k_{n}-1$ and
 $$A_{k_n}=\left\{\begin{array}{lcl}
     \{(2k_n)!,\ldots,(2k_n+1)!\} & , & (2k_n+1)!\leq n\\
      \{(2k_n)!,\ldots,n\} & , & (2k_n+1)!>n 
 \end{array}\right..$$

 Observe that for $1\leq \ell\leq k_n$,
 \begin{equation}
 \sum_{\{i,j\in A_{\ell}:i<j\}}e^{-\gamma(j-i)}=\sum_{i=1}^{\abs{A_{\ell}}-1}(\abs{A_{\ell}}-i)e^{-\gamma i}\leq\abs{A_{\ell}}E,
 \end{equation}
 where $E=\sum_{n=1}^{\infty}e^{-\gamma n}$. Let $C'>0$ be a constant with $E\leq C'V(X_0)$. Then
 \begin{equation}
 \sum_{\ell=1}^{k_n}\sum_{\{i,j\in A_{\ell}:i<j\}}e^{-\gamma(j-i)}\leq\sum_{\ell=1}^{k_n}\abs{A_{\ell}}E\leq C'\abs{S_1(n)}V(X_0).
 \end{equation}
 
 For the reminding terms we have
\begin{align*}
    \sum_{\ell=1}^{k_n}\sum_{m=1}^{\ell-1}\sum_{i\in A_m}\sum_{j\in A_{\ell}}e^{-\gamma(j-i)}&\leq\sum_{\ell=1}^{k_n}\sum_{m=1}^{\ell-1}\sum_{i\in A_m}\sum_{j\in A_{\ell}}e^{-\gamma((2\ell)!-i)}\\
    &=\sum_{\ell=1}^{k_n}\abs{A_{\ell}}\underbrace{\sum_{m=1}^{\ell-1}\sum_{i\in A_m}e^{-\gamma((2\ell)!-i)}}_{\leq E}\\
    &\leq \abs{S_1(n)}E\\
    &\leq C'\abs{S_1(n)}V(X_0).
\end{align*}
Putting everything together, we obtain
\begin{align*}
    V(S_n^*)&\leq\abs{S_1(n)}V(X_0)+4CC'\abs{S_1(n)}V(X_0)\leq(1+4CC')\abs{S_1(n)}V(X_0).
\end{align*}
Thus, our sequence of random variables satisfies $(ii)$ as well, and hence we conclude the SLLN, i.e. for $\mu$-almost every $x\in B_k$ we have
\begin{align*}
    \lim_{n\to\infty}\frac{1}{\abs{S_1(n)}}\sum_{i\in S_1(n)}\rho(L(x))&=\E(X_0)=\int\rho\dmu=\int-\log Q(x_0,x_1)\dmu(x)\\
    &=\sum_{w\in\LL_2(X)}-\mu([w])\log Q(w)\\
    &=\sum_{w\in\LL_2(X)}-{\bf q}(w_0)Q(w_0,w_1)\log Q(w_0,w_1)\\
    &=h(\mu).
\end{align*}
This completes the proof of the claim. 
\end{proof}

In order to complete the proof of Lemma \ref{SLLN} let $\iota\colon X\to X$ be the involution $\xi\mapsto\eta$, $\eta\mapsto\xi$, observe it is Lipschitz (similar argument as for $L$) and commute with $\sigma$. Since $\iota$ preserves cylinders of length 2, the measure $\iota_*\mu$ has the same integral as $\mu$ with respect to $\rho$. Using an analogous argument as for our previous limit, we obtain that for $\mu$-almost every $x\in B_k$
\begin{align*}
    \lim_{n\to\infty}\frac{1}{\abs{S_2(n)}}\sum_{i\in S_2(n)}\rho(\sigma^i(L(x)))&=\lim_{n\to\infty}\frac{1}{\abs{S_2(n)}}\sum_{i\in S_2(n)}\rho(\sigma^i(\iota(x)))\\
    &=\lim_{n\to\infty}\frac{1}{\abs{S_2(n)}}\sum_{i\in S_2(n)}\rho(\iota(\sigma^i(x)))\\
    &=\int\rho\circ\iota\dmu=\int\rho\,d(\iota_*\mu)\\
    &=h(\mu).
\end{align*}

Finally since $\abs{S_1(n)}+\abs{S_2(n)}=n$, we have
\begin{align*}
    \lim_{n\to\infty}-\frac{1}{n}\log\mu(C_{n+m}(L(x)))&=\lim_{n\to\infty}\frac{1}{n}\left(\sum_{i\in S_1(n)}\rho(\sigma^i(L(x)))+\sum_{i\in S_2(n)}\rho(\sigma^i(L(x)))\right)\\
    &=\lim_{n\to\infty}\left(\frac{\abs{S_1(n)}}{n}\frac{1}{\abs{S_1(n)}}\sum_{i\in S_1(n)}\rho(\sigma^i(L(x)))\right.\\
    &+\left.\frac{\abs{S_2(n)}}{n}\frac{1}{\abs{S_2(n)}}\sum_{i\in S_2(n)}\rho(\sigma^i(L(x)))\right)\\
    &=h(\mu)
\end{align*}
as desired.\qedhere
\end{proof}

\begin{proof}[Proof of Proposition \ref{entropyineq}] We follow the line of arguments as in \cite{PesinPitskel}. We are going to use the definition of entropy and notation from subsection \ref{DefEntropy}. Let $\delta>0$ and fix $m\in\N$. Consider the open cover $\UU$ of $X$ by cylinders of length $m$. Observe that for ${\bf U}\in\S$, the set $X({\bf U})$ is a cylinder of length $m+m({\bf U})$.

By Lemma \ref{SLLN}, given $0<\delta<\min\{h(\mu_k),1/2\}$ there is $D\subset B_k$ and $N\in\N$ such that $\mu_k(D)\geq 1-\delta$ and for every $x\in D$ we have
\begin{equation}\label{ineqGOD}
    \mu(C_{n+m}(L(x)))\leq e^{-n(h(\mu_k)-\delta)}
\end{equation}
for every $n\geq N$.

Let $\alpha<h(\mu_k)-\delta$ and denote the Carath\'eodory outer measure by $H:=m_C(\cdot,\alpha)$. Let $n\geq N$ and choose $\GG\subset\S$ a collection of strings that cover $L(B_k)$ with $m({\bf U})\geq n$ for every ${\bf U}\in\GG$ such that
\begin{equation}
    \Big|\inf\left\{\sum_{{\bf V}\in\JJ}e^{-m({\bf V})\alpha}\right\}-\sum_{{\bf U}\in\GG}e^{-m({\bf U})\alpha}\Big|\leq\delta,
\end{equation}
    where the infimum is taken over all $\JJ\subset\S$ that cover $L(B_k)$ with $m({\bf V})\geq n$ for every ${\bf V}\in\JJ$. Let $\GG_r:=\{{\bf U}\in\GG:m({\bf U})=r\}$ and $E_r:=\bigcup_{{\bf U}\in\GG_r}X({\bf U})$. Since $X({\bf U}')$ and $X({\bf U}'')$ are disjoint for distinct ${\bf U}',{\bf U}''$ by \eqref{ineqGOD} we have
    
\begin{equation}
    \mu_k(L(E_r)\cap D)=\sum_{{\bf U}\in\GG_r}\mu_k(L(X({\bf U}))\cap D)\leq\abs{\GG_r}e^{-r(h(\mu_k)-\delta)}.
\end{equation}

Then,
\begin{align*}
    \sum_{{\bf U}\in\GG}e^{-m({\bf U})\alpha}&=\sum_{r=n}^{\infty}e^{-r\alpha}\abs{\GG_r}\\
    &\geq\sum_{r=n}^{\infty}\mu_k(L(E_r)\cap D)e^{r(-\alpha+h(\mu_k)-\delta)}\\
    &\geq(1-\delta)e^{n(-\alpha+h(\mu_k)-\delta)}\\
    &\geq 1-\delta
\end{align*}
as $-\alpha+h(\mu_k)-\delta>0$.

It follows that
\begin{align*}
\inf\left\{\sum_{{\bf V}\in\JJ}e^{-m({\bf V})\alpha}:\JJ\text{ covers }L(B_k),m({\bf V})\geq n\right\}\geq 1-2\delta>0.
\end{align*}
Taking limit as $n\to\infty$ we obtain $H(L(B_k))=m_C(L(B_k),\alpha)>0$. By definition, we get $h_{\UU}(\sigma|L(B_k))\geq\alpha$. Since $\alpha<h(\mu_k)-\delta$ is arbitrary, we obtain $h_{\UU}(\sigma|L(B_k))\geq h(\mu_k)-\delta$. Now since $\delta$ is arbitrarily close to 0, we get $h_{\UU}(\sigma|L(B_k))\geq h(\mu_k)$. Finally we can send the length of the cylinders in the cover $\UU$ to infinity, and observing that $\diam\UU\to 0$ as $m\to\infty$ we get the desired result.\qedhere
\end{proof}

We now proceed to show that $N=\bigcup_{k\in\N}L(B_k)$ has full topological entropy. By Proposition \ref{propDimE} and Proposition \ref{entropyineq}, we have
\begin{align*}
    h_{top}(\sigma)&\geq h(\sigma|N)=\sup_{k\in\N}h(\sigma|L(B_k))\\
    &\geq\sup_{k\in\N}h(\mu_k)=\log\lambda\\
    &=h_{top}(\sigma)
\end{align*}
as desired.\hfill\qed

\subsection{Finding an uncountable collection $\NN$ of such sets $N$.}

We are going to find specific sequences $\QQ=(q_k)_{k\geq1}$ decreasing to 0 such that the corresponding sets $N_{\QQ}$ that they produce satisfy the conclusion of the main theorem.

Let $f,g\colon[0,1]\to\R$ be the following functions:

\begin{equation}
    f(q):=\left\{\begin{array}{lcc}
        \abs{\mu_q([\xi])-m([\xi])} & , & 0<q\leq 1  \\
         0 & , & q=0
    \end{array}\right.,
\end{equation}    
\begin{equation}
    g(q):=\left\{\begin{array}{lcc}
       \abs{\mu_q([\eta])-m([\eta])}  & , & 0<q\leq 1  \\
        0 & , & q=0 
    \end{array}\right..
\end{equation}

The functions $f$ and $g$ are continuous, and at least one of them must have positive values for some $q$ arbitrarily close to $0$. If not, then there is $q_0>0$ such that $f(q)=0$ and $g(q)=0$ for every $q\leq q_0$. This implies that for $q<q_0$ we have $\mu_q([\xi])=m([\xi])=m([\eta])=\mu_q([\eta])$, which contradicts the fact that $\mu_q([\xi])\neq\mu_q([\eta])$. Assume without loss of generality that the function $f$ has values of $q$ arbitrarily close to 0 where it is positive.

Choose now a strictly decreasing sequence $\EE=(\eps_k)_{k\geq 1}\subset(0,\max f]$ converging to 0. By the intermediate value theorem and the previous observation about $f$, there is a strictly decreasing sequence $\QQ(\EE)=(q_k)_{k\geq 1}$ converging to 0 such that $f(q_k)=\eps_k$. This sequence $\QQ(\EE)$ gives us a set $N_{\QQ(\EE)}$ satisfying properties (1), (2) and (4) in Theorem \ref{mainthm}.

Let $\EE'=(\eps_k')$ be another such sequence such that $\EE\cap\EE'=\emptyset$. 

\textit{Claim:} The sets $N_{\QQ(\EE)}$ and $N_{\QQ(\EE')}$ are disjoint.

\begin{proof}
    Let $y\in N_{\QQ(\EE)}\cap N_{\QQ(\EE')}$. Then, by definition of these sets, there are $\eps_k\neq\eps_{m}'$ such that $y\in L(B^{\QQ(\EE)}_k)\cap L(B^{\QQ(\EE')}_m)$. It follows that the Birkhoff average of $L(y)$ with respect to $\varphi$ converges to both $\mu_{q_k}([\xi])$ and $\mu_{q_m}([\xi])$, so these numbers must be equal. Thus, $\eps_k=f(q_k)=f(q_m)=\eps_m$, which is a contradiction.\qedhere 
\end{proof}

Therefore, in order to produce an uncountable collection of sets $N_{\QQ(\EE)}$, we need an uncountable collection of pairwise disjoint sequences $\EE$ converging to 0. One way of doing that is as follows. 

Let $A\subset(0,\infty)$ be an uncountable set of points which are pairwise rationally independent, i.e. for every $x,y\in A$ if $k_1x+k_2y=0$ with $k_1,k_2\in\Z$, then $k_1=k_2=0$. For every $a\in A$, consider the sequence
\begin{equation}
    \EE(a):=\left\{\frac{a}{m}:m\geq m_0\right\}=\{\eps_k(a):k\in\N\},
\end{equation}
starting from $m_0$ such that all the terms of the sequence are in $(0,\max f]$. This sequence gives us one of the sets we constructed $N_{\QQ(\EE(a))}$.

We claim that for $a\neq b\in A$ we have $\EE(a)\cap\EE(b)=\emptyset$. Indeed, if they have intersection, then there are integer numbers $m,m'$ such that
$$\frac{a}{m}=\frac{b}{m'},$$

therefore $m'a-mb=0$ which contradicts rationally independence of $a$ and $b$. Thus, the collection
$$\NN:=\{N_{\QQ(\EE(a))}:a\in A\}$$
is an uncountable collection of subsets of $X$ satisfying properties (1)--(4). This completes the proof of Theorem \ref{mainthm}.\hfill\qed

\end{document}